\documentclass[11pt,leqno]{article}

\usepackage{anysize}
\marginsize{3.5cm}{3.5cm}{2.5cm}{2.5cm}

\usepackage[]{hyperref}
\hypersetup{
    colorlinks=true,       
    linkcolor=red,          
    citecolor=blue,        
    filecolor=magenta,      
    urlcolor=cyan           
}

\usepackage{amsmath}
\usepackage{amsfonts,amssymb}
\usepackage{enumerate}
\usepackage{mathrsfs}
\usepackage{amsthm}
\usepackage{tikz}

\theoremstyle{plain}
\newtheorem{theo}{Theorem}[section]
\newtheorem*{thmcite}{Theorem}
\newtheorem{prop}[theo]{Proposition}
\newtheorem{lemm}[theo]{Lemma}

\theoremstyle{definition}
\newtheorem{rema}[theo]{Remark}

\DeclareSymbolFont{pletters}{OT1}{cmr}{m}{sl}
\DeclareMathSymbol{s}{\mathalpha}{pletters}{`s}

\def\eps{\varepsilon}

\def\mez{\frac{1}{2}}

\def\xR{\mathbf{R}}

\numberwithin{equation}{section}

\date{}
\title{A remark on quantitative unique continuation    from subsets of   the boundary of positive measure  }
  \author{Nicolas Burq and Claude Zuily}
\begin{document}
\maketitle
  \begin{abstract}The question of unique continuation of harmonic functions  in a domain $\Omega \subset \mathbb{R}^d$ with boundary $\partial \Omega$, satisfying Dirichlet boundary conditions and with normal derivatives vanishing  on a subset  $\omega$ of the boundary  is a classical problem. When $\omega$  contains an open subset of the boundary it is   a consequence of Carleman estimates (uniqueness for second order elliptic operators across an hypersurface). The case where $\omega$ is a set  of positive $(d-1)$ dimensional measure has attracted a lot of attention,  see e.g.~\cite{L, Ad-Esc1, To},  where qualitative results have been obtained in various situations.   The main open questions (about uniqueness) concern now Lipschitz domains and variable coefficients. Here, using results by Logunov and Malinnikova~\cite{LM1, LM2},  we consider the simpler case of $W^{2, \infty}$ domains but prove {\em quantitative} uniqueness both for {\em Dirichlet} and {\em Neumann} boundary conditions. As an application, we  deduce {\em quantitative} estimates for the Dirichlet and Neumann Laplace eigenfunctions on a $W^{2, \infty}$ domain with boundary.
  \end{abstract}

\section{Introduction}
The question discussed in this note starts with a famous   example by J. Bourgain and T. Wolff \cite{BW} saying that,

"in dimension $d \geq 3$, there exists a non trivial harmonic function  in the half space $\xR^d_+,$ which is $C^1$ up to the boundary, such that the function and its normal derivative vanish simultaneously on a set of positive measure of $\partial \xR^d_+".$

As emphasized by the authors the problem of constructing more regular counterexample ($C^2$ or $C^\infty$) is still open. Notice that this  example has been extended to arbitrary $C^{1, \alpha} $ domains in \cite{W}.

After that, the following  conjecture has been stated, which goes back to \cite {L}.

{\bf Conjecture}
Let $\Omega$ be a Lipschitz domain with boundary $\partial \Omega$ and let $\Sigma \subset \partial \Omega$ be an open subset. Let $u$ be a harmonic function in $\Omega$ which is continuous in $\overline{\Omega}.$ Suppose that $u$ vanishes in $\Sigma$ and its normal derivative $\partial_\nu u$ vanishes in a subset $\omega\subset \Sigma$ of strictly positive measure. Then $u \equiv 0$ in $\Omega.$

The same question can be, of course, asked for  elliptic operators with variable coefficients.

Several authors gave partial answer to this conjecture, see \cite {L}, \cite{AEK}, \cite{Ad-Esc1}, \cite{KN}. The most  recent result in this direction is due to    X. Tolsa \cite {To}   who gave, in the case of  the flat Laplacian, a positive answer to the conjecture for Lipschitz domains with small local Lipschitz constant.   However all   the above  papers are mainly qualitative and no estimate is given.

The purpose of this short note is to show that the deep  recent result by A. Logunov and E. Malinnikova \cite{LM2} can be used to give a quantitative form of this type of results for Laplace operators with respect to    Lipschitz metrics,   in the case  of open sets with $W^{2, \infty}$  ($C^{1,1}$) boundary. As an application we give estimates on the eigenfunctions of the Dirichlet or Neumann Laplace operator in the case where $\Omega$ is a relatively compact open set with $W^{2, \infty}$ boundary

\section{The results}
In a  domain  $\Omega \subset \xR^d$ with $W^{2, \infty}$ boundary    let  $g=(g_{ij}) $ be a  Riemanian metric which is assumed to be   locally Lipschitz in $\Omega.$ We denote by $g^{-1}= (g^{ij})$ the inverse metric and we denote by,
$$\Delta_g  = \frac{1}{\kappa(x)}\sum_{j,k= 1}^d \partial_j\big(\kappa(x)g^{jk}(x)\partial_k\big),$$ 
  the Laplace operator with respect to the   positive density $\kappa$ and to the positive definite  metric $g$.

If $Q \in \partial \Omega$ we denote by $B_r(Q)$ the ball in $\xR^d$ with center $Q$ and radius $r>0,$ and we set, 
$$D_r(Q) = B_r(Q)\cap \Omega, \quad \Gamma_r(Q) = B_r(Q)\cap \partial \Omega.$$
For a subset of $\omega \subset\partial \Omega$ we denote by $|\omega\vert_{d-1} $ its $(d-1)$ dimensional measure (here we endow $\partial \Omega$ with the metric $g\mid_{\partial \Omega}$). 
The purpose of this note is to show the following result.
\begin{theo}\label{result}
Let   $\Omega \subset \xR^d$ be an open set with $W^{2, \infty}$ boundary.  Let $m_0>0$.  There exist $r_0 >0,$ such that for any   $Q \in \partial \Omega$, $r \in (0,r_0)$,   there exist $C >0, \alpha \in (0,1)$ depending on $r, m_0$ such that for $m \geq m_0,$     for any  $\omega \subset \Gamma_r(Q)$,  with $|\omega|_{d-1} \geq m$, and any   $u \in W^{1,\infty}(\overline{\Omega})$   solution of $\Delta_gu = 0$ in $\Omega,$ such   that $ u = 0$ on $\Gamma_r(Q) $  we have,
$$\sup_{D_{\frac{r}{2}}(Q)} (\vert u \vert + \vert \nabla u \vert) \leq C\big(\sup_{\omega} \vert \partial_\nu u\vert\big)^\alpha \big(\sup_{D_r (Q)}   \vert \nabla u \vert\big)^{1-\alpha}.$$
\end{theo}

In the case of Neumann boundary conditions, we have also the following result.
\begin{theo}\label{result2}
Let   $\Omega \subset \xR^d$ be an open set with $W^{2, \infty}$ boundary.  Let $m_0>0$. There exist  $r_0 >0$ such that for any   $Q \in \partial \Omega$, $r \in (0,r_0)$, there exist $C >0, \alpha \in (0,1)$ depending on $r, m_0$ such that  for $m \geq m_0$      for any  $\omega \subset \Gamma_r(Q)$, with $|\omega|_{d-1} \geq m$  and any   $u \in W^{1,\infty}(\overline{\Omega})$   solution of $\Delta_gu = 0$ in $\Omega,$  such   that $ \partial_\nu u = 0$ on $\Gamma_r(Q) $  we have,
$$\sup_{D_{\frac{r}{2}}(Q)}  \vert \nabla u \vert  \leq C\big(\sup_{\omega} \vert \nabla u\vert\big)^\alpha \big(\sup_{D_r (Q)}   \vert \nabla u \vert\big)^{1-\alpha}.$$

\end{theo}
A natural question would be to ask for similar quantitative  results in the case of  Lipschitz boundary.

As a consequence of these {\em quantitative} results we get   estimates  of Laplace Dirichlet or Neumann eigenfunctions, both locally and globally. 
\begin{theo} \label{result3}Consider a $W^{2, \infty}$ relatively compact domain $\Omega$ endowed with a Lipschitz metric g and density $\kappa$. Consider the eigenfunctions of the Laplace operator on $\Omega$ with Dirichlet  (resp. Neumann)  boundary conditions,
$$ - \Delta_g e_n = \lambda_n ^2 e_n, \quad (\lambda_n  \geq 0), \qquad e_n  \arrowvert_{\partial\Omega} =0 \quad (\text{resp.  $\partial_\nu e_n \arrowvert_{\partial \Omega}  =0$}).
$$ 
Let  $Q \in \partial \Omega$ and $m_0>0$. There exists  $r_0 >0$ such that for any  $r \in (0,r_0)$, there exists $C, \alpha \in (0,1)$ depending on $r, m_0$ such that, with the above notations,  for any  $m \geq m_0,$   any $\omega \subset \Gamma_r(Q)$, with $ |\omega|_{d-1} \geq m$, and  any~$n$,

Dirichlet case:
\begin{equation}\label{dirichlet}
  \sup_{D_{\frac{r}{2}}(Q)} (\vert e_n  \vert + \vert \nabla e_n  \vert)  \leq Ce^{\lambda_n r} \big(\sup_{\omega} \vert \partial_\nu e_n\vert\big)^\alpha \big(\lambda_n \sup_{D_r (Q)}   \vert  e_n \vert + \sup_{D_r (Q)}   \vert \nabla e_n \vert\big)^{1-\alpha},
\end{equation}
  Neumann case:
\begin{equation}\label{neumann}
  \sup_{D_{\frac{r}{2}}(Q)}   \vert \nabla e_n  \vert   \leq Ce^{\lambda_n r} \big( \lambda_n \sup_{\omega} \vert e_n\vert+ \sup_{\omega} \vert\nabla e_n\vert\big)^\alpha \big(\lambda_n \sup_{D_r(Q)} \vert e_n\vert  + \sup_{D_r (Q)}    \vert \nabla e_n \vert\big)^{1-\alpha}.  
\end{equation}
\end{theo}
\begin{theo}[Global high frequency estimate]\label{global}
Let $\Omega$ be a $W^{2, \infty}$ domain with boundary, $\omega \subset\partial\Omega$ of positive (d-1)-Lebesque measure. Then there exists $C>0$ such that for any eigenfunction $e$  of the Laplace operator with either Dirichlet or Neumann boundary condition associated to the eigenvalue $-\lambda^2, \, (\lambda \geq 1)$, we have,
\begin{equation} 
\| e \|_{L^2(\Omega)} \leq C e^{C  \lambda } (\|  e \Vert_{L^1( \omega)} +   \Vert \vert \nabla e \|_{L^1( \omega)}).  
 \end{equation}
\end{theo}
\begin{rema} At least for Dirichlet boundary conditions, it would be possible to relax the smoothness assumption of $\partial \Omega$ to Lipschitz with an additional  star-shaped assumption allowing to employ Carleman boundary estimates (see~\cite{LR95}), while still assuming that the regularity in a neighborhood of $\omega$ is $W^{2, \infty}$ (see e.g.~\cite{AEWZ} for similar results in the context of the heat equation, with analytic additional assumption near~$\omega$). For simplicity we do not pursue this track in this note
\end{rema}
 \section{  Proofs}

We shall use the following result  proved in \cite{BM}.  
\begin{prop}[\cite{BM}  Proposition 3.3]\label{propBM}
Assume that $\Omega$ is a domain in $\xR^d$ with $W^{2,\infty}$ boundary. Let $g= (g_{jk})$ be a Lipschitz Riemannian metric and $\kappa$ be a positive Lipschitz density in $\Omega.$ Denoting by $g^{-1} = (g^{ij})$ the inverse matrix of $g$ we set,
$$\Delta_g = \frac{1}{\kappa(x)} \sum_{j,k=1}^d \partial_j\big(\kappa(x) g^{jk}(x)\partial_k).$$
Then near any point $Q \in \partial \Omega$ there exists a $W^{2, \infty}$ diffeomorphism which sends $Q$ to the origin and transforms  $D_r(Q), \Gamma_r(Q)$ for small $r>0$, $\omega \subset \Gamma_r(Q)$ and $\Delta_g$ to,
\begin{align*}
 &\widetilde{D} = \{(x',x_d): \vert x'\vert < \delta', x_d \in (0, \eps)\}, \quad  \widetilde{\Gamma} =  \{(x',x_d): \vert x'\vert < \delta', x_d=0\},\\
 & \widetilde{\omega} \subset \widetilde{\Gamma}, \quad \text{measure}( \widetilde{\omega})  >0,\\
 & \Delta_{\widetilde{g}} = \frac{1}{\widetilde{\kappa}(x',x_d) } \text{div}_{(x'x_d)}\,\big(\widetilde{\kappa}(x',x_d) \widetilde{g}(x',x_d)\nabla_{(x'x_d)}\big),\\
 &\widetilde{g}(x',0)= 
 \begin{pmatrix}
 \widetilde{g}_0(x')& 0\\
 0& 1
 \end{pmatrix}
 \end{align*}
  where $\widetilde{g}$ is a Lipschitz metric in $\widetilde{D}$ and $\widetilde{g}_0(x')$ is a $(d-1)\times (d-1)$ symmetric matrix.
\end{prop}
\begin{proof}[Proof of Theorem \ref{result}]
If $u$ is a solution of $\Delta_g u = 0$ in $D_r(Q)$ then $\widetilde{u},$  its transform by the diffeomorphism is a solution of $\Delta_{\widetilde{g}} \widetilde{u}= 0$ in $\widetilde{D}.$ By the classical regularity theorem for elliptic equations we know that $\widetilde{u}$ is $W^{2,\infty}$ in $\widetilde{D}.$  Moreover recall that by hypothesis we have $\widetilde{u}(x',0)=0.$

Let us set, for $(x',x_d)$ such   that $\vert x'\vert < \delta', x_d \in (-\eps, \eps)$,
\begin{equation}\label{ext}
v(x',x_d) =
 \left\{
 \begin{array}{rl}
\widetilde{u}(x',x_d), & \quad \text{if } x_d \in (0, \eps), \\ -\widetilde{u}(x',-x_d),& \quad \text{if }  x_d \in (-\eps,0).
\end{array}
\right.
\end{equation}
It is easy to see that $v \in W^{2,\infty}(\mathcal{O})$ where, 
$$\mathcal{O} = \{(x',x_d): \vert x'\vert < \delta', x_d \in (-\eps, \eps)\}.$$
Let us set, for $(x',x_d)\in \mathcal{O}$, 
\begin{equation}\label{rho}
\rho(x',x_d)=
\left\{
\begin{array}{rl}
\widetilde{\kappa}(x', x_d), &\quad \text{if } x_d \in (0, \eps),\\ \widetilde{\kappa}(x', -x_d),&\quad \text{if } x_d \in (-\eps,0).
\end{array}
\right.
\end{equation}
Moreover for $1 \leq j,k \leq d-1$ or $j=k=d$ let us set,
\begin{equation}\label{ht}
h_{jk}(x',x_d)=
\left\{
\begin{array}{rl}
\widetilde{g}_{jk}(x', x_d), &\quad \text{if } x_d \in (0, \eps) ,\\ \widetilde{g}_{jk}(x', -x_d),&\quad \text{if } x_d\in (-\eps,0),
\end{array}
\right.
\end{equation}
Now we set, for $1 \leq j \leq d-1$,
\begin{equation}\label{hmixte}
h_{jd}(x',x_d)   =
\left\{
\begin{array}{rl}
 \widetilde{g}_{jd}(x', x_d)  \quad &\text{if } x_d \in (0, \eps) ,\\  -\widetilde{g}_{jd}(x', -x_d)  \quad &\text{if } x_d\in (-\eps,0),
\end{array}
\right.
\end{equation}
and the same formula for $h_{dj}, 1 \leq j \leq d-1.$
Notice that the function $\rho$ and $h_{jk}, 1 \leq j,k\leq d$ are Lipschitz in $\mathcal{O}.$ This is obvious for $\rho$ and $h_{jk}$  for $1 \leq j,k \leq d-1$ or $j=k=d.$ For $h_{jd}$ and $h_{dj}$ this follows from the fact that according to Proposition \ref{propBM} we have, $\widetilde{g}_{jd}(x',0) = \widetilde{g}_{dj}(x',0) =0.$

Eventually we set,
$$ P= \frac{1}{\rho(x',x_d)} \text{div}_{(x',x_d)}  \big( \rho(x',x_d) h(x',x_d)\nabla_{(x',x_d)}\big).$$
Notice that the operator $P$ is still elliptic in $\mathcal{O}$ since, if we denote by $\sigma_2(P)$ its principal symbol  we have,
\begin{equation*}
 \sigma_2(P)(x,\xi) =
 \left\{
 \begin{array}{rl}
   \sigma_2(\Delta_{\widetilde{g}})(x',x_d, \xi',\xi_d) \quad &\text{if } x_d \in(0, \eps),\\ \sigma_2(\Delta_{\widetilde{g}})(x',-x_d, \xi',-\xi_d) \quad &\text{if } x_d \in (-\eps,0).   
   \end{array}
   \right.
\end{equation*}
\begin{lemm}
We have, 
$$ Pv = 0 \text{ in } \mathcal{O}.$$
\end{lemm}
\begin{proof}
We have,
\begin{align*}
\rho P  &= P_1 +P_2 +P_3 \quad \text{where},\\
P_1&= \sum_{j,k=1}^{d-1} \partial_j(\rho h_{jk}\partial_k),\quad P_2 = \sum_{j=1}^{d-1}\big[ \partial_j(\rho h_{jd}\partial_d) +   \partial_d(\rho h_{dj}\partial_j)\big], \quad P_3 =   \partial_d(\rho h_{dd}\partial_d). 
 \end{align*}
Since $P_1$ has only tangential  derivatives we have,
\begin{equation}\label{P1}
P_1 v =
\left\{
\begin{array}{rl}
 \sum_{j,k=1}^{d-1} \big[\partial_j(\rho \widetilde{g}_{jk}\partial_k)\widetilde{u}\big](x',x_d) &\text{if } x_d \in (0, \eps),\\ - \sum_{j,k=1}^{d-1} \big[\partial_j(\rho\widetilde{g}_{jk}\partial_k)\widetilde{u}\big](x',-x_d) &\text{if } x_d \in (-\eps,0).
\end{array}
\right.
\end{equation}
Now according to \eqref{hmixte} we have, for $1 \leq j \leq d-1$, 

\begin{equation}\label{ext2}
\partial_j(\rho h_{jd}\partial_d v)(x',x_d)) =
 \left\{
 \begin{array}{rl}
\partial_j(\widetilde{\kappa} \widetilde{g}_{jd}\partial_d \widetilde{u})(x',x_d), & \quad \text{if } x_d \in (0, \eps), \\   - \partial_j(\widetilde{\kappa} \widetilde{g}_{jd}\partial_d \widetilde{u})(x',-x_d) &\quad \text{if }  x_d \in (-\eps,0),
\end{array}
\right.
\end{equation}
and an analogue formula for the term involving $h_{dj}$.

Eventually according to \eqref{ht} we have,
\begin{equation}\label{P3}
P_3 v =
\left\{
\begin{array}{rl}
  \big[\partial_d(\rho \widetilde{g}_{dd}\partial_d)\widetilde{u}\big](x',x_d) &\text{if } x_d \in (0, \eps),\\    -\big[\partial_d(\rho\widetilde{g}_{dd}\partial_d)\widetilde{u}\big](x',-x_d) &\text{if } x_d \in (-\eps,0).
\end{array}
\right.
\end{equation}
It follows from  \eqref{P1}, \eqref{ext} and \eqref{P3} that,
\begin{equation*}
\rho Pv =
\left\{
\begin{array}{rl}
\Delta_{\widetilde{g}} \widetilde{u}(x',x_d)\quad & \text{if } x_d \in (0, \eps),\\ -\Delta_{\widetilde{g}} \widetilde{u}(x',-x_d)\quad & \text{if } x_d \in (-\eps,0),
\end{array}
\right.
\end{equation*}
thus $Pv =0$ as claimed.
\end{proof}
Now we are in position to apply \cite[Theorem 5.1]{LM2}. Indeed $P$ is an elliptic operator in divergence form with Lipschitz coefficients in the set $\mathcal{O}$, $\widetilde{\omega}$ is a subset of $\mathcal{O}$ with strictly positive $(d-1)$ Lebesgue measure and $v$ is a solution of $Pv = 0$ in $\mathcal{O}.$  By this Theorem we can infer that, for every compact in $\mathcal{O}$  there exist  $C>0$ and $\alpha \in (0,1)$ independent of $v$ such that,
$$\sup_{K}\vert \nabla v \vert \leq C \big(\sup_{\widetilde{\omega}}\vert \nabla v \vert\big)^\alpha \big(\sup_{\mathcal{O}}\vert \nabla v \vert\big)^{1-\alpha}.$$
Now since by hypothesis $v=0$ on $\mathcal{O} \cap \{x_d=0\}$ the tangential derivatives of $v$ vanish on this set, so we are left with the normal derivative of $v$ on $\widetilde{\omega}.$ Moreover  an elementary Poincare inequality shows that we can estimate the $L^\infty$ norm  of $v$ by that of $\nabla v.$ Restricting ourselves to $x_d >0$ and going back to $u$ by the diffeomorphism we obtain the result in Theorem \ref{result}.
 \end{proof}
 \begin{proof}[Proof of Theorem \ref{result2}]
The proof of Theorem \ref{result2} is completely analogous. Instead of \eqref{ext} we have just to set,
  for $(x',x_d)$ such that $\vert x'\vert < \delta', x_d \in (-\eps, \eps)$,
\begin{equation}\label{ext3}
v(x',x_d) =
 \left\{
 \begin{array}{rl}
\widetilde{u}(x',x_d), & \quad \text{if } x_d \in (0, \eps), \\ \widetilde{u}(x',-x_d),& \quad \text{if }  x_d \in (-\eps,0).
\end{array}
\right.
\end{equation}
\end{proof}
 \begin{proof}[Proof of Theorem \ref{result3}]
 We shall apply Theorem \ref{result} in the following context. Let us set $\widetilde{\Omega} = \xR_t \times \Omega$ which is $(d+1)$ dimensional. Then $\partial \widetilde{\Omega} = \xR_t \times \partial \Omega.$ Let $ \widetilde{Q} = (0, Q), Q \in \partial\Omega, \widetilde{B}_r = \{(t,x): \vert t \vert + \vert x \vert<r\}, \widetilde{D}_r= \widetilde{B}_r\cap \widetilde{\Omega},
 \widetilde{\omega} = (-r,r)\times \omega.$ Then $ \widetilde{\omega}$ is a subset of $\partial \widetilde{\Omega}$  with positive  $d$-Lebesgue measure.
 
 We consider first the Dirichlet case. Consider the function $u_n(t,x) = e^{\lambda_n t} e_n(x)$ which solves the  elliptic equation $(\partial_t^2 + \Delta_g)u_n = 0$  in $\widetilde{\Omega}.$ Then we apply Theorem \ref{result}. Since $\nabla_x e_n = \nabla_x u_n(0,x)$ and the normal derivative to $\partial \widetilde{\Omega}$ is still  $\partial_\nu$ we obtain,  
\begin{align*} 
  \sup_{D_\frac{r}{2}} \vert\nabla_x e_n\vert   &\leq \sup_{\widetilde{D}_\frac{r}{2}} \vert\nabla_{x} u_n\vert\\
  & \leq C\big( \lambda_n e^{\lambda_n r}  \sup_{D_r}\vert  e_n\vert +  e^{\lambda_n r}\sup_{D_r}\vert \nabla_x e_n\vert\big)^{1-\alpha} \big( e^{\lambda_n r}   \sup_{\omega}\vert \partial_\nu e_n\vert\big)^\alpha,\\
  & \leq C e^{\lambda_n r}  \big( \lambda_n\sup_{D_r}\vert  e_n\vert + \sup_{D_r} \vert \nabla_x e_n\vert\big)^{1-\alpha} \big(    \sup_{\omega}\vert\partial_\nu e_n\vert\big)^\alpha.
   \end{align*}
   The estimate in the Neumann case is completely analogous. 
  \end{proof}
 
\begin{proof}{Proof of Theorem~\ref{global}.} 
 We are going to deduce it from the local version Theorem~\ref{result3}.
  By compactness, we can cover $\partial \Omega$ by a finite number of balls 
  $$ \partial \Omega = \cup_{j=1}^N \Gamma_{r_0 (Q_j)}(Q_j), \quad \Gamma_{r_0 (Q_j)}(Q_j)= B(Q_j, r_0(Q_j))\cap \partial \Omega, \quad Q_j \in \partial \Omega.$$
 Setting $D(Q_j, r_0(Q_j)) = B(Q_j, r_0(Q_j))\cap \Omega$ we deduce 
   from~\cite[Theorem 14.6]{JeLe99}  that there exists $C,D$ depending only on $M$ such that for any $j\in \{1, \dots, N\}$, 
 \begin{equation}\label{jerison-lebeau}
  \| e \|_{L^2(\Omega)}\leq C e^{D\lambda} \| e\|_{L^2(D(Q_j, \mez r_0(Q_j)))}.
  \end{equation}
 Remark that in ~\cite{JeLe99}, \eqref{jerison-lebeau} is stated (and proved) only on compact manifolds without boundary. To prove it in our case we shall use   the following result which allows to reduce the study to the case where there is no boundary.
 

\begin{thmcite}[The double manifold~\protect{\cite[Theorem 7]{BM}}]\label{double}
Let $g$ be given. There exists a $W^{2, \infty}$ structure on the double manifold $\widetilde{M}$, a metric $\widetilde{g}$  of class $W^{1, \infty}$ on $\widetilde{M}$, and  a density $\widetilde{\kappa}$ of class $W^{1, \infty}$ on $\widetilde{M}$ such that the following holds.
\begin{itemize}
\item The maps
$$ i^\pm  x\in M  \rightarrow (x, \pm 1) \in \widetilde{M} = M \times \{\pm 1\} / \partial M$$ are isometric embeddings.
\item The density induced on each copy of $M$ is the density $\kappa$,
$$\widetilde{\kappa} \mid_{M\times\{\pm1 \}} = \kappa.
$$
\item For any eigenfunction $e$ with eigenvalue $\lambda^2$ of the Laplace operator $-\Delta = - \frac 1 {\kappa} \text{div } g^{-1} \kappa \nabla$ with Dirichlet or Neumann boundary conditions, there exists an eigenfunction $\widetilde{e}$ with the same eigenvalue $\lambda$ of the Laplace operator $-\Delta = - \frac 1 {\widetilde{\kappa} }\text{div } \widetilde{g}^{-1} \widetilde{\kappa} \nabla$ on $\widetilde{M}$ such that, 
\begin{equation}\label{ext}
\widetilde{e} \mid_{M \times \{1\}} = e, \quad \widetilde{e} \mid_{M \times \{-1\}} = \begin{cases}  -e \quad &(\text{Dirichlet boundary conditions}),\\
 e \qquad &(\text{Neumann boundary conditions}). \end{cases}
\end{equation}
\end{itemize}
\end{thmcite}
On the other hand, there exists $j_0, 1\leq j_0, \leq N$ such that
 $$|\omega \cap \Gamma_{r_0 (Q_{j_0})}(Q_{j_0})\vert_{d-1} \geq \frac{|\omega|_{d-1}} {N} .$$
 For simplicity we denote now by $r_0 =r_0 ( Q_{j_0})$ and by $Q= Q_{j_0}$.
 Combining~\eqref{jerison-lebeau} with~\eqref{dirichlet} or~\eqref{neumann}, we get
 \begin{equation}\label{eq12}
  \| e\| _{L^2(\Omega)} \leq Ce^{\lambda (r_0+D)} \big( \lambda \sup_{\omega\cap D_{r_0}(Q)} \vert e\vert+ \sup_{\omega\cap D_{r_0}(Q)} \vert\nabla e\vert\big)^\alpha \big(\lambda \sup_{D_{r_0}(Q)} \vert e\vert  + \sup_{D_{r_0} (Q)}    \vert \nabla e \vert\big)^{1-\alpha}.
  \end{equation}
To eliminate the right hand side  in the  above inequality, we now use~\cite[Proposition 2.1]{BM}
\begin{prop}\label{sobolev}There exists $\sigma>0$ such that with, 
$$ \mathcal{H}^\sigma= D( (- \Delta)^{\sigma/2})$$ endowed with its natural norm, 
$$\| u\|_{\mathcal{H}^\sigma}= \Bigl(\sum_{k} |u_k |^2 (1+ \lambda_k)^{2\sigma} \Bigr)^{1/2} ,
$$
we have,
$$\| \nabla_x u \|_{L^\infty}+  \| u \|_{L^\infty} \leq C \| u \|_{\mathcal{H}^\sigma},$$
where $\mathcal{H}^\sigma$ is the domain of the operator $(-\Delta )^{\frac 1 2 } $ and consequently, 
$$ \|u\|_{\mathcal{H}^\sigma} = \|((- \Delta)^\sigma+ \text{Id} ) u \|_{L^2}.$$
\end{prop}
It follows that, 
$$\| \nabla_x e \|_{L^\infty}+  \| e\|_{L^\infty} \leq C \Vert e \Vert_{\mathcal{H}^\sigma} \leq C'(1+ \lambda^\sigma)\Vert e \Vert_{L^2(\Omega)}.$$

As a consequence, we get from~\eqref{eq12},
$$ \| e\| _{L^2(\Omega)} \leq Ce^{\lambda (r_0+D)} \big( \lambda \sup_{\omega\cap D_{r_0}(Q)} \vert e\vert+ \sup_{\omega\cap D_{r_0}(Q)} \vert\nabla_x e \vert\big)^\alpha \big((1+ \lambda) (\lambda^\sigma+1) \big)^{1-\alpha}\Vert e\Vert^{1-\alpha}_{L^2(\Omega)}.$$
 Since $\lambda \geq 1$ we obtain eventually,
$$ \|e\|^\alpha _{L^2(\Omega)} \leq C (1+ \lambda)^{(1- \alpha) \sigma +1} e^{\lambda (r_0+D)}\big(  \sup_{\omega} \vert e\vert+ \sup_{\omega} \vert\nabla_x e\vert \big)^\alpha
$$
and consequently (we can assume $r_0 \leq 1$)
\begin{equation}\label{eq.14}
\begin{aligned}
 \|e\| _{L^2(\Omega)} &\leq C^{\frac 1 \alpha }  (1+ \lambda)^{\frac{{(1- \alpha) \sigma +1} }{\alpha}} e^{\lambda\frac{ (1+D)} \alpha} \big( \sup_{\omega} \vert e\vert+ \sup_{\omega} \vert\nabla _x e\vert\big) \\
 &\leq K e^{M\lambda} \big( \sup_{\omega} \vert e\vert+ \sup_{\omega} \vert\nabla_x e \vert\big)
 \end{aligned}
 \end{equation}
 
To conclude it remains to replace the sup norm in the right hand side  of~\eqref{eq.14} by an $L^2$ norm. To do so, we use that according to~\cite[Theorem 5.1]{LM2}  the constants $C, \alpha$ appearing in~\eqref{dirichlet}, \eqref{neumann} (and hence also the constants $K, M$ in~\eqref{eq.14}) remain bounded independently of $m,$ as long as $|\omega|_{d-1} \geq m$ with $m \geq m_0 >0$.
Let us take $|\omega|_{d-1} \geq m \geq 2m_0$ and set $K_0= K(\frac m {2}), M _0 = M (\frac{m} {2})$.
Let, 
\begin{align*}
&F= \{ x\in \omega : |e (x)\vert \leq \frac 1 {4K_0}e^{-M_0\lambda}\| e\|_{L^2(\Omega)}\},\\
&  G= \{ x\in \omega :  |\nabla_x e(x)\vert \leq  \frac 1 {4K_0} e^{-M_0\lambda}\| e\|_{L^2(\Omega)}.\} 
\end{align*}
and $E= F\cap G$. If,
  $$ |E| \geq \frac m {2} $$ 
then from ~\eqref{eq.14} (applied with $\omega$ replaced by $E$), 
$$ \|e\| _{L^2(\Omega)} 
 \leq K_0 e^{M_0\lambda} \big( \sup_{E} \vert e\vert+ \sup_{E} \vert\nabla_x e\vert\big)\leq\frac 1 2 \| e\|_{L^2(\Omega)},
 $$ which is absurd. We deduce that, 
 $$ |E| \leq \frac m {2}.$$
 (Notice that this case contains in particular the case where $E = \emptyset$).
 
 Since $\vert \omega\vert \geq m$ we deduce that  $\vert E^c \cap \omega\vert \geq \frac{m}{2}.$ Now $E^c \cap \omega = \big(F^c \cap \omega\big) \cup \big(G^c \cap \omega \big)$
 which implies either $|F^c\cap \omega | \geq \frac m 4$, or $|G^c \cap \omega| \geq \frac m 4$.  In the first case we can write,
 $$\int_\omega \vert e(x)\vert\, d\sigma \geq \int_{F^c \cap \omega} \vert e(x)\vert\, d\sigma \geq \frac{m}{4}\frac 1 {4K_0}e^{-M_0\lambda}\| e\|_{L^2(\Omega)},$$
 so
 $$\| e\|_{L^2(\Omega)} \leq \frac{16 K_0}{m} e^{M_0 \lambda}\int_\omega \vert e(x)\vert\, d\sigma,$$
 and in the second case we get,
 $$\| e\|_{L^2(\Omega)} \leq \frac{16 K_0}{m} e^{M_0 \lambda}\int_\omega \vert \nabla_xe(x)\vert\, d\sigma,$$

 Therefore in both cases we get, 
 $$ \| e\|_{L^2(\Omega)} \leq \frac{16 K_0}{m} e^{M_0 \lambda} \int_{\omega} (|e| + |\nabla_xe|)d\sigma .$$
  
\end{proof}

Let us end this section with a final remark (which also follows from Logunov-Malinnikova's results) about Bourgain and Wolf counter example~\cite{BW}
\begin{rema} The counter example of Bourgain and Wolf cannot be extended locally accross the boundary $\partial \mathbb{R}^d_+$ as a harmonic function 
\end{rema} 
Indeed, if it was possible to extend it as a harmonic function, it would vanish, as well as its normal derivative  on a set of $\omega$ of positive $(d-1)$ Lebesgue measure. However, almost every point in $\omega$ is a Lebesgue point and for all such Lebesgue points, $x_0$, it is easy to see that the tangential gradient also vanishes, because for all tangential directions $w$, we can find a seqnence $y_n \in  \omega$ such that 
$$ \lim_{ n\rightarrow + \infty}\frac {y_n- x_0}{ \| y_n - x_0\|}= w ,$$
which implies $\frac{ \partial u} { \partial w} =0$. We deduce that $\nabla_x u$ vanishes on a subset of $\partial \mathbb{R}^d$ positive $(d-1)$  Lebesgue measure, and consequently $\nabla_x u=0$.

\end{document}